\documentclass[12pt,a4paper]{amsart}
\usepackage{amssymb} 
\usepackage{mathrsfs} 
\usepackage{newtxtext}
\usepackage[varg]{newtxmath}
\usepackage[shortlabels]{enumitem} 
\usepackage{bm}
\usepackage[colorlinks=false]{hyperref}
\usepackage{graphicx, xcolor}
\usepackage{constants}
\newtheorem{theorem}{Theorem}[section]

\newtheorem{proposition}[theorem]{Proposition}
\newtheorem{lemma}[theorem]{Lemma}

\newtheorem{definition}[theorem]{Definition}
\newtheorem{remark}[theorem]{Remark}

\newtheorem{prop}[theorem]{Proposition}
\newtheorem{lem}[theorem]{Lemma}

\DeclareMathOperator{\Div}{div}
\DeclareBoldMathCommand{\bu}{u}
\DeclareBoldMathCommand{\bv}{v}
\newcommand{\unknown}{\rho}
\let\unk\unknown
\newcommand{\F}{\mathcal{F}}
\newcommand{\D}{\mathcal{D}}
\newcommand{\N}{\mathbb{N}}
\newcommand{\R}{\mathbb{R}}
\numberwithin{equation}{section}
\usepackage{mathtools}
\mathtoolsset{showonlyrefs=true}

\usepackage{amsrefs}
\title{Long-time behavior of free energy in the nonlinear Fokker-Planck equation}
\author{Kouta Araki}
\address{Department of Mathematics,  College of Science and Technology,  Nihon University,  Tokyo
101-8308 JAPAN}
\email{csku24001@g.nihon-u.ac.jp}

\author{Masashi Mizuno}
\address{Department of Mathematics,  College of Science and Technology,  Nihon University,  Tokyo
101-8308 JAPAN}
\email{mizuno.masashi@nihon-u.ac.jp}
\subjclass[2020]{Primary~35B40, Secondary~35A15, 35A09, 35B09, 35K20, 35K55, 35K65, 35Q84}
\keywords{Nonlinear Fokker-Planck equation; Porous medium equation; Entropy dissipation methods; Long-time asymptotic bahavior}
\allowdisplaybreaks[1]
\begin{document}

% temporary command: DELETE for final version
%\linenumbers

\begin{abstract}
    We study the asymptotic behavior of Fokker-Planck equations with spatially inhomogeneous nonlinear diffusion, based on the energy dissipation law.
    First, we consider the Fokker-Planck equation with porous-medium-type nonlinear diffusion that satisfies the energy dissipation law by introducing spatial inhomogeneity into the free energy.
    We obtain a result on the long-time behavior of the dissipation function for sufficiently large diffusion coefficients by extending the entropy dissipation method to the case of inhomogeneous diffusion.
\end{abstract}

\maketitle

\section{Nonlinear Fokker-Planck model with Inhomogeneous Diffusion}
\label{Inhomogeneous}

Let $\Omega \subset \R^n$ be a bounded convex domain with smooth
boundary in the $n$-dimensional Euclidean space,  $\nu$ be the outer unit
normal vector on $\partial \Omega$.  Let $\alpha>1$ be a constant. We
consider the following initial-boundary value problem for the
nonlinear Fokker-Planck equation.
\begin{equation}
\label{Nonlinear-Fokker-Planck}
\tag{NFP}
\left\{ \, 
 \begin{aligned}
  \frac{\partial \unknown}{\partial t} - \Div(\unknown\nabla(\alpha d(x)\unknown^{\alpha-1}+\phi(x))) &= 0, \quad &x \in \Omega, \quad t>0, \\
  \unknown(0, x) &= \unknown_0(x), \quad &x \in \Omega ,  \\
  \unknown\nabla(\alpha d(x) \unknown^{\alpha-1}+\phi(x)) \cdot \nu&=0, \qquad &x \in \partial \Omega ,  \quad t>0.
 \end{aligned}
\right.
\end{equation}
Here $d,  \phi,  \unknown_0$ are given $C^2$ functions on
$\overline{\Omega}$. We assume that there exists a positive constant
$\Cl{const:D_min}>0$ such that
\begin{equation}
 d(x)\ge \Cr{const:D_min}.
\end{equation}
Assume $\unknown_0=\unknown_0(x)\colon\overline{\Omega}\rightarrow \R$
be a given positive probability density function on $\overline{\Omega}$, 
namely
\begin{equation}
 \label{eq:f_0 is pdf}
  \int_\Omega \unknown_0\, dx=1.
\end{equation}
If $d$ is a positive constant,  then
\begin{equation}
 \label{porous medium}
  \Div(\unknown\nabla(\alpha d \unknown^{\alpha-1}))=(\alpha-1)d\Delta \unknown^{\alpha}=\Div((\alpha-1)d\nabla\unknown^\alpha)
\end{equation}
is valid. Thus,  \eqref{Nonlinear-Fokker-Planck} is widely known as a
drift-diffusion equation with porous medium type diffusion. On the other
hand,  if $d$ is not a constant,  all three terms in \eqref{porous medium}
are different. Why do we consider \eqref{Nonlinear-Fokker-Planck}? We
first explain the motivation to study \eqref{Nonlinear-Fokker-Planck}.

\subsection{Energy dissipation law with linear diffusion}
When $d$ is a positive constant, the Fokker-Planck equation of the form
\begin{equation}
 \label{eq:1.linear-Fokker-Planck}
 \frac{\partial\unknown}{\partial t}
  -d\Delta\unknown
  -
  \Div(\unknown\nabla\phi(x))
  =
  0
\end{equation}
is related to the following stochastic differential equation
\begin{equation}
 \label{eq:1.SDE}
 dX=-\nabla\phi(X)\, dt+d\, dB, 
\end{equation}
where $B$ is a Brownian motion. Precisely,  if $\{X_t\}_{t>0}$ is a
solution of \eqref{eq:1.SDE},  then associated stochastic density
function $\rho$ satisfies \eqref{eq:1.linear-Fokker-Planck} in
distribution sense. Note that \eqref{eq:1.linear-Fokker-Planck} can be written
as
\begin{equation}
 \label{eq:1.ContinuityEquation-LinearHomogeneousDiffusion}
 \frac{\partial\unknown}{\partial t}
  -
  \Div(\unknown\nabla(d\log\unknown+\phi(x)))
  =
  0,
\end{equation}
hence we obtain the energy dissipation law
\begin{equation}
 \label{eq:1.EnergyDissipationLaw-LinearHomogeneousDiffusion}
 \frac{d}{dt}
  \int_{\Omega}
  (d(\log\unknown-1)+\phi(x))\unknown\, dx
  =
  -
  \int_\Omega
  |\nabla(d\log\unknown+\phi(x))|^2\unknown
  \, dx
\end{equation}
for solutions $\rho$ of \eqref{eq:1.linear-Fokker-Planck} subjected to
the natural boundary condition
\begin{equation}
    \unknown\nabla(d\log\unknown+\phi(x))\cdot\nu=0\ \text{on}\ \partial\Omega.    
\end{equation}

Next,  we look at the spatial inhomogeneity of the diffusion. We return
to the stochastic differential equation \eqref{eq:1.SDE} with spatially
variable diffusion, namely
\begin{equation}
 \label{eq:1.SDE-VariableDiffusion}
 dX=-\nabla\phi(X)\, dt+d(X) dB.
\end{equation}
Then,  we need to specify the stochastic integration to determine
\eqref{eq:1.SDE-VariableDiffusion}. For instance,  if we choose It\={o}'s
integral,  then the associated Fokker-Planck equation is
\begin{equation}
 \label{eq:1.linear-Fokker-Planck-SDE}
  \frac{\partial\unknown}{\partial t}
  -\Delta(d(x)\unknown)
  -
  \Div(\unknown\nabla\phi(x))
  =
  0.
\end{equation}
Compare to \eqref{eq:1.linear-Fokker-Planck},  it is not easy to find
the energy dissipation law
\eqref{eq:1.EnergyDissipationLaw-LinearHomogeneousDiffusion} for
\eqref{eq:1.linear-Fokker-Planck-SDE}. As in
\eqref{eq:1.ContinuityEquation-LinearHomogeneousDiffusion},  we can
rewrite \eqref{eq:1.linear-Fokker-Planck-SDE} as
\begin{equation}
 \frac{\partial\unknown}{\partial t}
  -
  \Div(\unknown
  (d(x)\nabla\log\unknown+\nabla d(x)+\nabla\phi(x))
  )
  =
  0, 
\end{equation}
and one can find that the velocity vector
$-d(x)\nabla\log\unknown-\nabla d(x)-\nabla\phi(x)$ does not have a
scalar potential function in general. Note that we can formulate other
equations of the \eqref{eq:1.SDE-VariableDiffusion} from other
stochastic integrals (for instance,  Stratonovich's integral), but similar
difficulties occur for any stochastic integral.

Our idea to guarantee the energy dissipation law with spatial inhomogeneity
is not to start with a stochastic differential equation \eqref{eq:1.SDE} but
\eqref{eq:1.EnergyDissipationLaw-LinearHomogeneousDiffusion} with
the spatial inhomogeneity. Let us consider
\begin{equation}
 \label{eq:1.EnergyDissipationLaw-LinearInhomogeneousDiffusion}
 \frac{d}{dt}
  \int_{\Omega}
  (d(x)(\log\unknown-1)+\phi(x))\unknown\, dx
  =
  -
  \int_\Omega
  |\nabla(d(x)\log\unknown+\phi(x))|^2\unknown
  \, dx.
\end{equation}
Since $\rho$ is a probability density function,  we consider the equation
of continuity
\begin{equation}
 \label{eq:1.ContinuityEquation}
 \frac{\partial\unknown}{\partial t}
  +
  \Div(\unknown\vec{v})
  =
  0
\end{equation}
where $\vec{v}$ is a velocity vector. Plugging
\eqref{eq:1.ContinuityEquation} into
\eqref{eq:1.EnergyDissipationLaw-LinearInhomogeneousDiffusion},  we
obtain
\begin{equation}
 \int_\Omega
  \nabla(d(x)\log\unknown+\phi(x))\cdot\vec{v}
  \unknown
  \, dx
  =
  -
  \int_\Omega
  |\nabla(d(x)\log\unknown+\phi(x))|^2\unknown
  \, dx.
\end{equation}
Thus,  we find $\vec{v}=-\nabla(d(x)\log\unknown+\phi(x))$ in order to
guarantee energy dissipation law
\eqref{eq:1.EnergyDissipationLaw-LinearInhomogeneousDiffusion}. Plugging
$\vec{v}$ into the equation of continuity
\eqref{eq:1.ContinuityEquation},  we obtain
\begin{equation}
 \label{eq:1.InhomFP-LinDiff}
 \frac{\partial\unknown}{\partial t}
  +
  \Div(
  \unknown
  (\nabla(d(x)\log\unknown+\phi(x)))
  )
  =
  0.
\end{equation}

\subsection{Energy dissipation law with nonlinear diffusion}

We are in replacing the linear diffusion $\Delta\unknown$
\eqref{eq:1.linear-Fokker-Planck} to the nonlinear diffusion
$\Delta\unknown^\alpha$ of the porous medium type (cf. \cite{MR2286292}). 
Let us consider the
energy dissipation law with the free energy including the spatial
inhomogeneity of the form
\begin{equation}
 \label{eq:1.EnergyDissipationLaw-NLInhomogeneousDiffusion}
  \begin{split}
   \frac{d}{dt}\F[\unknown](t)
   &=
   -
   \D[\unknown](t), 
   \\
   \F[\unknown](t)
   &:=
   \int_\Omega
   (\alpha d(x)\unknown^{\alpha-1}+\phi(x))\unknown\, dx, 
   \\
   \D[\unknown](t)
   &:=
   \int_\Omega
   |\nabla(\alpha d(x)\unknown^{\alpha-1}+\phi(x))|^2
   \unknown\, dx.
  \end{split}
\end{equation}
As the same argument,  we plug \eqref{eq:1.ContinuityEquation} into
\eqref{eq:1.EnergyDissipationLaw-NLInhomogeneousDiffusion} and obtain
\begin{equation}
 \int_\Omega
  \nabla(\alpha d(x)\unknown^{\alpha-1}+\phi(x))\cdot
  \vec{v}\unknown\, dx
 =
 -
 \int_\Omega
 |\nabla(\alpha d(x)\unknown^{\alpha-1}+\phi(x))|^2
 \unknown\, dx.
\end{equation}
In order to guarantee the energy dissipation law
\eqref{eq:1.EnergyDissipationLaw-NLInhomogeneousDiffusion},  we take
\begin{equation}
    \vec{v}=-\nabla(\alpha d(x)\unknown^{\alpha-1}+\phi(x)).    
\end{equation}
Plugging
$\vec{v}$ into the equation of continuity
\eqref{eq:1.ContinuityEquation},  we obtain the first equation of
\eqref{Nonlinear-Fokker-Planck}.
Note that for the case of homogeneous diffusion, \cite{MR1842429} gave 
a physical derivation of the porous medium equation, similar to this argument.

\subsection{Properties of the Nonlinear Fokker-Planck equation}
Let
\begin{equation}
 \label{mu}
  \mu:=\alpha d(x)\unknown^{\alpha-1}+\phi(x).
\end{equation}
Then $\vec{v}=-\nabla\mu$ and \eqref{Nonlinear-Fokker-Planck} can be
rewritten as
\begin{equation}
 \label{continuity_eq}
  \frac{\partial \unknown}{\partial t}-\Div(\unknown\nabla\mu)=0.
\end{equation}

We first give a notion of solutions of \eqref{Nonlinear-Fokker-Planck}.

\begin{definition}
 $C^2$ positive function $\unknown$ on $\overline{\Omega}$ is a
 classical solution of \eqref{Nonlinear-Fokker-Planck} if $\unknown$
 satisfies \eqref{Nonlinear-Fokker-Planck} in classical sense.
\end{definition}

Since \eqref{Nonlinear-Fokker-Planck} comes from the equation of
continuity,  we can show the conservation of mass.

\begin{lemma}
 \label{lem;f is PDF}
 Let $\unknown_0$ be a positive probability density
 function on $\overline{\Omega}$ and let $\unknown$ be a positive
 classical solution of \eqref{Nonlinear-Fokker-Planck}. Then,  for any
 $t>0$
 \begin{equation}
  \label{eq; f_is_PDF}
   \int_\Omega \unknown(x, t)\, dx=1.
 \end{equation}
\end{lemma}

\begin{proof}
 By the integration by parts together with
 \eqref{Nonlinear-Fokker-Planck},  we obtain
 \begin{equation*}
  \frac{d}{dt}\int_\Omega \unknown\, dx=\int_\Omega \unknown_t\, dx
   =\int_\Omega \Div(\unknown\nabla\mu)\, dx
   =\int_{\partial \Omega}\unknown \nabla\mu\cdot\nu\, d\sigma
   =0.
 \end{equation*}
 This follows
 \begin{equation*}
  \int_\Omega \unknown(x, t)\, dx=\int_\Omega \unknown_0(x)\, dx=1.
 \end{equation*}
\end{proof}

Next,  recall that $\F$ can be written as
\begin{equation}
 \label{free energy}
  \F[\unknown](t)=\int_\Omega (d(x)\unknown^\alpha+\unknown\phi(x))\, dx.
\end{equation}
Then,  we can establish the energy dissipation law for
\eqref{Nonlinear-Fokker-Planck}.

\begin{proposition}
 \label{prop;diff_type_energy_law}
 Let $\unknown_0$ be a positive probability density function on
 $\overline{\Omega}$ and let $\unknown$ be a positive classical solution
 of \eqref{Nonlinear-Fokker-Planck}. Let $\F$ be the free energy defined
 as \eqref{free energy}. Then,  for any $t>0$
 \begin{equation}
  \label{prop;diff_type_energy_law1}
   \frac{d}{dt}\F[\unknown](t)=-\int_\Omega |\nabla\mu|^2\unknown\, dx\le 0.
 \end{equation}
\end{proposition}

\begin{proof}
 Consider to time-derivative of $\F$,  then we obtain
  \begin{equation*}
   \begin{split}
    \frac{d}{dt}\F[\unknown](t)&=\frac{d}{dt}\int_\Omega(d(x)\unknown^\alpha+\unknown\phi(x))\, dx \\
    &=\int_\Omega \frac{\partial}{\partial t}(d(x)\unknown^\alpha+\unknown\phi(x))\, dx\\
    &=\int_\Omega \unknown_t(\alpha d(x)\unknown^{\alpha-1}+\phi(x))\, dx.
   \end{split}
  \end{equation*}
 Plugging \eqref{Nonlinear-Fokker-Planck} to $\rho_t$ with \eqref{mu}, 
 we have
 \begin{equation*}
  \int_\Omega \unknown_t(\alpha d(x)\unknown^{\alpha-1}+\phi(x))\, dx
   =
   \int_\Omega \Div(\unknown\nabla\mu)\mu\, dx
 \end{equation*}
 Then,  integration by parts with the boundary condition
 \eqref{Nonlinear-Fokker-Planck} deduce
 \begin{equation*}
  \int_\Omega \Div(\unknown\nabla\mu)\mu\, dx
   =\int_\Omega \Div(\unknown (\nabla\mu) \mu)\, dx
   -
   \int_\Omega |\nabla\mu|^2 \unknown\, dx 
   =
   -\int_\Omega |\nabla\mu|^2\unknown\, dx.
 \end{equation*}
 hence we obtain \eqref{prop;diff_type_energy_law1}.
\end{proof}

\begin{remark}
    The inequality \eqref{prop;diff_type_energy_law1} means that the free energy $\F$ is a Lyapunov functional for solutions of \eqref{Nonlinear-Fokker-Planck}.
    We refer to \cites{MR0760592,MR0760591} to derive a Lyapunov functional for solutions to the self-similar transform of the porous medium equation.
\end{remark}

Integrating both side of \eqref{prop;diff_type_energy_law1} with $t \in
[0, T]$,  the following integral-type energy dissipation law holds;
\begin{equation}
 \label{eq;integral-type-dissipation_law}
  \F[\unknown](t)
  +
  \int_0^T \int_\Omega |\nabla\mu|^2\unknown\, dxdt
  =
  \F[\unknown_0].
\end{equation}
Since $d, \unknown\ge0$,  
\begin{equation*}
 \F[\unknown](t)=\int_\Omega (d(x)\unknown^\alpha+\unknown\phi(x))\, dx
 \ge
 \int_\Omega \unknown\phi(x)\, dx
 \geq
 -\|\phi\|_\infty
\end{equation*}
hence from \eqref{eq;integral-type-dissipation_law},  we find
\begin{equation}
  \int_0^T \int_\Omega |\nabla\mu|^2\unknown\, dxdt
  \le
  \F[\unknown_0]+\|\phi\|_\infty.
\end{equation}

Recall that $\D$ can write by using $\mu$ as
\begin{equation}
 \D[\rho](t) := \int_\Omega|\nabla\mu|^2\unknown\, dx.
\end{equation}
From \eqref{eq;integral-type-dissipation_law},  we can show asymptotic
behavior of $\D$ sequentially in time.

\begin{lemma}
 \label{lem;pointwise} Let $\unknown_0$ be a positive probability
 density function on $\overline{\Omega}$.  Let $\unknown$ be a positive
 global-in-time classical solution of
 \eqref{Nonlinear-Fokker-Planck}. Assume $\F[\rho_0]<\infty$.
 Then,  there is an increasing sequence
 $\{t_j\}_{j \in \N}$,  such that $t_j \rightarrow \infty$ and
 \begin{equation}
  \label{eq;pointwise}
   \int_\Omega |\nabla\mu|^2\unknown\, dx \rightarrow 0, \quad j\rightarrow\infty
 \end{equation}
\end{lemma}

\begin{proof}
 From \eqref{eq;integral-type-dissipation_law} and $F[\unknown_0]<\infty$,  we have
 \begin{equation*}
  \int_0^\infty \int_\Omega |\nabla \mu|^2\unk \,  dxdt 
   \leq
   \F[\rho_0]+\|\phi\|_\infty
   <\infty.
 \end{equation*}
 Thus,  there is an increasing sequence $\{t_j\}_{j \in \N}$, such that
 $t_j \rightarrow \infty $,  and
 \begin{equation*}
  \int_\Omega |\nabla \mu|^2\unknown\, dx \rightarrow 0, \quad j \rightarrow \infty.
 \end{equation*}
\end{proof}

From Lemma \ref{lem;pointwise}, we raise the following problem. Can we
show the full convergence of the dissipation function $\D$ in time,  namely
\begin{equation}
 \label{full_limit}
  \D[\rho](t)
  =
  \int_\Omega
  |\nabla\mu|^2\unknown\, dx
  \rightarrow 0
\end{equation}
as $t\rightarrow\infty$? This question is related to the long-time behavior
of $\unknown$ to the equilibrium state. From \eqref{full_limit},  we can
expect $\nabla\mu\rightarrow0$ as $t\rightarrow\infty$. Then, the solution
$\unknown$ may converge to the equilibrium state $\unknown_\infty$, 
which satisfies
\begin{equation}
 \alpha d(x)\unknown_\infty(x)^{\alpha-1}+\phi(x)
  =\Cl{const:equilibrium}.
\end{equation}
Here $\unknown_\infty$ is determined by a constant
$\Cr{const:equilibrium}$ to be a probability density function. We are
interested in the long-time behavior of the solution $\unk$ of
\eqref{Nonlinear-Fokker-Planck} to the equilibrium state $\unk_\infty$. 

\subsection{Known results}

When $d$ is constant, and $\phi$ is a strongly convex function,  we can
employ the entropy dissipation method
\cites{MR1853037, MR1777035, MR3497125, MR1842429}. The main idea of the entropy
dissipation method is to compute the second time derivative of
$\F[\rho]$ and show
\begin{equation}
 \frac{d^2}{dt^2}\F[\unk](t)
  =
  -\frac{d}{dt}\D[\unk](t)
  \geq
  \Cl{const:ConstDiff-Carrillo}
  \D[\unk](t)
\end{equation}
for some positive constant $\Cr{const:ConstDiff-Carrillo}>0$. Then we
obtain exponential decay of $\D[\rho]$ by the Gronwall theorem. Applying
the Csisz\'{a}r-Kullback-Pinsker inequality to show the long-time
asymptotic behavior in $L^1$ space. 

When $d$ is not constant,  to the best of our knowledge,  there is no
result about long-time asymptotic behavior for
\eqref{Nonlinear-Fokker-Planck}. There are a few results about the study
of long-time asymptotics with the variable diffusion coefficient in
\cite{MR1853037}; however, the problem is completely different from the
model \eqref{Nonlinear-Fokker-Planck}. We mention the recent study by
\cites{MR4547562, 
MR4506846, MR4976469, epshteyn2025longtimeasymptoticbehaviornonlinear}. In
these papers,  one considered free energy,  dissipation function,  and the
energy dissipation law of the form
\eqref{eq:1.EnergyDissipationLaw-LinearHomogeneousDiffusion}. To ensure
the energy dissipation,  we may deduce \eqref{eq:1.InhomFP-LinDiff}.
Long-time asymptotics of \eqref{eq:1.InhomFP-LinDiff} subjected to
the periodic boundary condition were studied by
\cites{MR4506846, epshteyn2025longtimeasymptoticbehaviornonlinear}, and
well-posedness of \eqref{eq:1.InhomFP-LinDiff} was studied by
\cites{MR4976469, MR4547562}. Note that these works are related to the
study of the stochastic model of grain boundary motion \cite{MR4526584}.

The problem \eqref{Nonlinear-Fokker-Planck} is quite a different setting
in contrast with the previous study \eqref{eq:1.InhomFP-LinDiff}. First, 
the energy dissipation law with the free energy $\F$ defined as
\eqref{eq:1.EnergyDissipationLaw-NLInhomogeneousDiffusion} deduces the
nonlinear diffusion,  in contrast with the linear diffusion
\eqref{eq:1.InhomFP-LinDiff}. Further, we consider the Neumann boundary
condition in \eqref{Nonlinear-Fokker-Planck},  compare with the periodic
boundary condition in
\cites{MR4506846, epshteyn2025longtimeasymptoticbehaviornonlinear}.

\subsection{Main Theorem}

Here we state the main theorem.

\begin{theorem}
\label{theorem}
    Let $n=1, 2, 3$. Let $\unknown$ be a bounded strictly positive global-in-time classical solution of \eqref{Nonlinear-Fokker-Planck} on $\overline{\Omega}$,  namely there are positive constants $\Cl{const:f>c}$ and $\Cl{const:f<c}>0$ such that
 \begin{equation}
  \label{eq;apripri}
   \Cr{const:f>c} \le \unknown(x, t) \le \Cr{const:f<c}
 \end{equation}
 for all $x \in \overline{\Omega}$ and $t>0$. Assume that there is a
 positive constant $\lambda>0$ such that $\nabla^2\phi \ge \lambda I$, 
 where $I$ is the identity matrix and $\nabla^2\phi$ is Hesse matrix of
 $\phi$. In addition,  assume that $\nabla d$,  $\nabla \phi$ are bounded
 on $\overline{\Omega}$. Then,  there are positive constants
 $\Cr{const:D_min}$,  $\Cl{const:Initial_Energy}$,  
 $\Cl{const:Dissipation_Decay}>0$ depending only on $n$, $\lambda$, 
 $\Omega$, $\alpha$, $\|\nabla d\|_{L^\infty(\Omega)}$, $\|\nabla
 \phi\|_{L^\infty(\Omega)}$, $\Cr{const:f>c}$,  $\Cr{const:f<c}$ such
 that if two conditions
    \begin{equation}
     \label{eq:1.TheoremAssumption}
     \min_{x \in \overline{\Omega}}d(x)\ge \Cr{const:D_min}, \quad 
      \D[\unk_0]
      =
      \int_\Omega |\nabla \mu(x, 0)|^2\unknown_0\, dx
      \le \Cr{const:Initial_Energy}
    \end{equation}
    hold,  then
    \begin{equation}
    \label{exp}
     \D[\unk](t)
     =
     \int_\Omega |\nabla\mu|^2 \unknown\, dx \le \Cr{const:Dissipation_Decay}e^{-\lambda t},  \qquad t>0.
    \end{equation}
\end{theorem}

Theorem \ref{theorem} says that even though $\nabla d$ is large,  we obtain exponential
decay of $\D[\rho](t)$ if the diffusion coefficient $d$ is sufficiently
large. Note that if $\nabla d=0$, namely $d$ is constant,  we can take
$\Cr{const:D_min}$ arbitrary positive number. We do not know whether the
assumption \eqref{eq:1.TheoremAssumption},  especially the lower bounds
of $d$, is essential or not. We also mention that the assumption
$n=1, 2, 3$ is used to apply the Sobolev inequality.  

In particular,  from \eqref{exp} we have
\begin{equation*}
 \int_\Omega |\nabla\mu|^2\unknown\, dx\rightarrow 0, \quad t \rightarrow \infty
\end{equation*}
for sufficient large $d(x)$.

We briefly explain the proof of the main theorem. First,  as the
same argument in \cite{MR3497125},  we follow the entropy dissipation
method. Compute the second time derivative of free energy $\F[\rho]$. We
have new terms from the spatial derivative of the diffusion coefficient
$d$. Next,  we treat the integrals of the spatial derivative of $d$. We
have two types of integrals: One has quadratic $\nabla\mu$; the
other has cubic $\nabla\mu$. The integral of quadratic $\nabla\mu$ can
be controlled by the dissipation function and the integral of the second
derivative of $\mu$ by using the H\"older and Young inequalities. To
treat the integral of cubic $\nabla\mu$,  we use the Sobolev-Poincar\'e
inequality and the interpolation inequality. The dimension assumption
$n=1, 2, 3$ is needed to make the interpolation inequality. The assumption
\eqref{eq:1.TheoremAssumption} is to control the opposite coefficient
of the dissipation function.

\subsection{Notation}

Let $\Omega\subset\R^n$ be an open set and let $f\colon\Omega\rightarrow\R$ be
a sufficiently smooth function $f:\Omega\rightarrow\R$. We denote the
gradient of $f$ as
\begin{equation}
 \nabla f
  :=
  \left(
   \frac{\partial f}{\partial x_1}, \frac{\partial f}{\partial x_2}, \dots, \frac{\partial f}{\partial x_n}
   \right).
\end{equation}
We denote the Hesse matrix of $f$ as
\begin{equation}
 \nabla^2f:=
  \begin{pmatrix}
         \frac{\partial^2 f}{\partial x_1^2 } &\cdots &\frac{\partial^2 f}{\partial x_1\partial x_i}& \cdots&\frac{\partial^2 f}{\partial x_1 \partial x_n}\\
         \vdots & \ddots& & &\vdots \\
         \frac{\partial^2 f}{\partial x_i\partial x_1}&\cdots&\frac{\partial^2 f}{\partial x_i^2}&\cdots&\frac{\partial^2 f}{\partial x_i \partial x_n}\\
         \vdots& & &\ddots&\vdots\\
         \frac{\partial^2 f}{\partial x_n\partial x_1}&\cdots
         &\frac{\partial^2 f}{\partial x_n\partial x_i}&\cdots&\frac{\partial^2 f}{\partial x_n^2}
  \end{pmatrix}.
\end{equation}
The Laplacian of $f$ is denoted as
\begin{equation}
 \Delta f
  :=
  \sum_{i=1}^n \frac{\partial^2 f}{\partial x_i^2}.
\end{equation}
For $n$-dimensional symmetric matrices $X, Y$,  we define $X\le Y$ to be
the case that for all $\xi \in \R^n$
\begin{equation*}
 X \xi \cdot \xi \le Y \xi \cdot\xi.
\end{equation*}
We denote $I$ the $n$-dimensional identity matrix. Thus,  for
$n$-dimensional symmetric matrix $X$,  $cI\leq X$ for some $c\in\R$ means
that the eigenvalue of $X$ is equal or greater that $c$.

\section{Proof of main theorem}
\label{proof_of_Main_theorem}
Exponential decay \eqref{exp} is demonstrated by evaluating the second time derivative of $\F$ from below using the dissipation function.
By direct computation, we have
\begin{equation}
    \label{d^2F/dt^2_0}
    \begin{split}
        \frac{d^2}{dt^2}\F[\unknown](t)
        &=
        \frac{d}{dt}\left(-\int_\Omega |\nabla\mu|^2\unknown\, dx\right)
        \\
        &=
        -\int_\Omega |\nabla\mu|^2\unknown_t\, dx-2\int_\Omega (\nabla\mu\cdot\nabla\mu_t)\unknown\, dx.
    \end{split}
\end{equation}

We compute the first term of \eqref{d^2F/dt^2_0} in the right-hand side.

\begin{lem}
    \label{|nablamu|^2f_t}
    Let $\unknown$ be a classical solution of \eqref{Nonlinear-Fokker-Planck} on $\overline{\Omega}\times [0, \infty)$. Then, 
    \begin{equation}
    \label{|nablamu|^2f_t_1}
        \begin{split}
            -\int_{\Omega}|\nabla\mu|^2\unknown_t\, dx
            &=
            2\alpha\int_\Omega (\nabla\mu\cdot \nabla^2(d(x)\unknown^{\alpha-1})\nabla\mu)\unknown\, dx
            \\
            &\quad
            +2\int_\Omega (\nabla\mu\cdot\nabla^2\phi(x)\nabla\mu)\unknown\, dx.
        \end{split}            
    \end{equation}
\end{lem}
\begin{proof}
    Using the integration by parts and \eqref{Nonlinear-Fokker-Planck},  we obtain that
\begin{equation*}
   -\int_\Omega |\nabla\mu|^2\unknown_t\, dx=-\int_\Omega |\nabla\mu|^2\Div(\unknown\nabla\mu)\, dx
            =\int_\Omega (\nabla(|\nabla\mu|^2) \cdot\nabla\mu)\unknown\, dx.
\end{equation*}
Next,  we compute $\nabla(|\nabla \mu|^2)\cdot\nabla\mu$. We denote $\nabla\mu=(\mu_{x_1}, \mu_{x_2}, \dots \mu_{x_n})$. Then,  by direct calculation,  we obtain that
    \begin{equation*}
        \begin{split}(\nabla(|\nabla\mu|^2)\cdot\nabla\mu)&=\sum_{i=1}^n\left(\sum_{j=1}^n\mu_{x_j}^2\right)_{x_i}\mu_{x_i} \\
            &=\sum_{i, j=1}^n 2\mu_{x_j}\mu_{x_jx_i}\mu_{x_i}\\
            &=2\sum_{j=1}^n \mu_{x_j}\sum_{i=1}^n \mu_{x_jx_i}\mu_{x_i} \\
            &=2(\nabla\mu\cdot\nabla^2\mu\nabla\mu).
        \end{split}
    \end{equation*}
    Since $\mu=\alpha d(x) \unknown^{\alpha-1}+\phi(x)$,  we obtain \eqref{|nablamu|^2f_t_1}
%    \begin{equation*}
%        -\int_\Omega |\nabla\mu|^2\unknown_t\, dx=2\alpha\int_\Omega(\nabla\mu\cdot\nabla^2(d(x)\unknown^{\alpha-1})\nabla\mu)\unknown\, dx+2\int_\Omega (\nabla\mu\cdot\nabla^2\phi\nabla\mu)\unknown\, dx.
%    \end{equation*}
\end{proof}

Next,  we compute the second term of \eqref{d^2F/dt^2_0} in the right-hand side.

\begin{lem}
    \label{nablamu_t}
    Let $\unknown$ be a classical solution of \eqref{Nonlinear-Fokker-Planck} on $\overline{\Omega}\times [0, \infty)$. Then, 
    \begin{equation}
    \label{nablamu_t1}
        \begin{split}
            -\int_{\Omega}(\nabla\mu\cdot\nabla\mu_t)\unknown\, dx
            &=
            \alpha(\alpha-1)\int_\Omega d(x)(\nabla \unknown\cdot\nabla\mu)^2\unknown^{\alpha-2}\, dx \\
            &\quad+2\alpha(\alpha-1)\int_\Omega d(x)(\nabla \unknown\cdot \nabla\mu)\Delta\mu \unknown^{\alpha-1}\, dx\\
            &\quad+\alpha(\alpha-1)\int_\Omega d(x) (\Delta\mu)^2\unknown^\alpha\, dx.
        \end{split}
    \end{equation}
\end{lem}
\begin{proof}
    Since \eqref{mu} and \eqref{Nonlinear-Fokker-Planck},  we obtain that
    \begin{equation}
    \label{eq:nablamu_t}
    \nabla\mu_t=\alpha (\alpha-1)\nabla(d(x)\unknown^{\alpha-2}\unknown_t)=\alpha (\alpha-1)\nabla(d(x)\unknown^{\alpha-2}\Div(\unknown\nabla\mu)).
    \end{equation}
    Using integration by parts together with the boundary condition of \eqref{Nonlinear-Fokker-Planck},  we have
    \begin{equation*}
        \begin{split}
            -\int_\Omega (\nabla\mu\cdot\nabla\mu_t)\unknown\, dx
             &=
             -\alpha(\alpha-1)\int_\Omega (\nabla\mu\cdot\nabla(d(x)\unknown^{\alpha-2}\Div(\unknown\nabla\mu)))\unknown\, dx\\
            &=\alpha(\alpha-1)\int_\Omega d(x)\unknown^{\alpha-2}(\Div(\unknown\nabla\mu))^2\, dx\\
            &=\alpha(\alpha-1)\int_\Omega d(x)(\nabla \unknown\cdot\nabla\mu)^2\unknown^{\alpha-2}\, dx\\
            &\quad
            +2\alpha(\alpha-1)\int_\Omega d(x)(\nabla \unknown\cdot\nabla\mu)\Delta\mu \unknown^{\alpha-1}\, dx\\
            &\quad
            +\alpha(\alpha-1)\int_\Omega d(x)(\Delta\mu)^2\unknown^\alpha\, dx.
        \end{split}
    \end{equation*}
\end{proof}
Plugging \eqref{|nablamu|^2f_t_1} and \eqref{nablamu_t1} to \eqref{d^2F/dt^2_0},  we obtain
\begin{equation}
\label{d^2F/dt^2_3}
    \begin{split}
        \frac{d^2}{dt^2}\F[\unknown](t)&=-\int_\Omega |\nabla\mu|^2\unknown_t\, dx-2\int_\Omega (\nabla\mu\cdot\nabla\mu_t)\unknown\, dx\\
        &=
        2\int_\Omega (\nabla\mu\cdot\nabla^2\phi(x)\nabla\mu)\unknown\, dx
        \\
        &\quad
        +2\alpha\int_\Omega (\nabla\mu\cdot\nabla^2(d(x)\unknown^{\alpha-1})\nabla\mu)\unknown\, dx \\
        &\quad
        +2\alpha(\alpha-1)\int_\Omega d(x)(\nabla \unknown\cdot\nabla\mu)^2\unknown^{\alpha-2}\, dx
        \\
        &\quad
        +4\alpha(\alpha-1)\int_\Omega d(x)(\nabla \unknown\cdot \nabla\mu)\Delta\mu \unknown^{\alpha-1}\, dx\\
        &\quad+2\alpha(\alpha-1)\int_\Omega d(x) (\Delta\mu)^2\unknown^\alpha\, dx.
    \end{split}
\end{equation}

We prepare the following lemma to estimate the $\nabla^2 (d(x)\unknown^{\alpha-1})$ term of \eqref{d^2F/dt^2_3}
\begin{lem}
    \label{nabla^2(d(x)f^(alpha-1))}
    Let $\unknown$ be a classical solution of \eqref{Nonlinear-Fokker-Planck} on $\overline{\Omega}\times [0, \infty)$. Then, 
    \begin{equation}
    \label{nabla^2(d(x)f^(alpha-1))1}
        \begin{split}
            \int_\Omega (\nabla\mu\cdot\nabla^2(d(x)\unknown^{\alpha-1})\nabla\mu)\unknown\, dx
            &=
            -(\alpha-1)\int_\Omega d(x)(\nabla \unknown\cdot\nabla^2\mu\nabla\mu)\unknown^{\alpha-1}\, dx
            \\
            &\quad-(\alpha-1)\int_\Omega d(x)(\nabla \unknown\cdot\nabla\mu)\Delta\mu \unknown^{\alpha-1}\, dx
            \\
            &\quad
            -(\alpha-1)\int_\Omega d(x)(\nabla \unknown\cdot\nabla\mu)^2\unknown^{\alpha-2}\, dx
            \\
            &\quad
            -\int_\Omega (\nabla d(x)\cdot\nabla^2\mu\nabla\mu)\unknown^\alpha\, dx 
            \\
            &\quad
            -\int_\Omega (\nabla d(x)\cdot\nabla\mu)\Delta\mu \unknown^\alpha\, dx
            \\
            &\quad-\int_\Omega (\nabla d(x)\cdot\nabla\mu)(\nabla \unknown\cdot\nabla\mu)\unknown^{\alpha-1}\, dx.
            \end{split}
    \end{equation}
\end{lem}
\begin{proof}
    We compute $\nabla^2(d(x)\unknown^{\alpha-1})$. We denote 
    \begin{equation}
        \nabla^2 (d(x)\unknown^{\alpha-1})=((d(x)\unknown^{\alpha-1})_{x_ix_j})_{i, j}. 
    \end{equation}
    Then,  by direct calculations,  we obtain
    \begin{equation*}
        \begin{split}(\nabla\mu\cdot\nabla^2(d(x)\unknown^{\alpha-1})\nabla\mu)\unknown&=\sum_{i=1}^n\left(\mu_{x_i}\left(\sum_{j=1}^n(d(x)\unknown^{\alpha-1})_{x_ix_j}\mu_{x_j}\right)\right)\unknown\\
            &=\sum_{i, j=1}^n\left(\mu_{x_i}(d(x)\unknown^{\alpha-1})_{x_i}\mu_{x_j}\unknown\right)_{x_j}
            \\
            &\quad
            -\sum_{i, j=1}^n\left((d(x)\unknown^{\alpha-1})_{x_i}(\mu_{x_i}\mu_{x_j}\unknown)_{x_j}\right).
        \end{split}
    \end{equation*}
    The first term of the right-hand side turns into
    \begin{equation*}
        \sum_{i, j=1}^n\left(\mu_{x_i}(d(x)\unknown^{\alpha-1})_{x_i}\mu_{x_j}\unknown\right)_{x_j}=\Div((\nabla\mu\cdot\nabla(d(x)\unknown^{\alpha-1})\unknown\nabla\mu).
    \end{equation*}
    By calculating the second term of the right-hand side,  we obtain
    \begin{equation*}
        \begin{split}
            &\qquad
            -\sum_{i, j=1}^n\left((d(x)\unknown^{\alpha-1})_{x_i}(\mu_{x_i}\mu_{x_j}\unknown)_{x_j}\right)
            \\
            &=
            -\sum_{i, j=1}^n\left((d(x)\unknown^{\alpha-1})_{x_i}
            \left(
            (
            \mu_{x_i x_j}\mu_{x_j})\unknown
            +
            \mu_{x_i}\mu_{x_jx_j}\unknown
            +
            \mu_{x_i}\mu_{x_j}\unknown_{x_j}
            \right)
            \right)\\
            &=
            -(\nabla(d(x)\unknown^{\alpha-1})\cdot\nabla^2\mu\nabla\mu)\unknown
            -(\nabla(d(x)\unknown^{\alpha-1})\cdot\nabla\mu)\Delta\mu \unknown
            \\
            &\quad
            -(\nabla(d(x)\unknown^{\alpha-1})\cdot\nabla\mu)(\nabla\mu\cdot\nabla \unknown)
            .
        \end{split}
    \end{equation*}
    Thus,  we obtain
    \begin{equation}
    \label{eq:2.Hesse-Diffusion}
        \begin{split}
        &\quad
        (\nabla\mu\cdot\nabla^2(d(x)\unknown^{\alpha-1})\nabla\mu)\unknown
        \\
        &=
        \Div(\nabla\mu\cdot\nabla(d(x)\unknown^{\alpha-1})\unknown\nabla\mu)
        \\
        &\quad
        -\nabla(d(x)\unknown^{\alpha-1})\cdot
        \left(
        \unknown\nabla^2\mu\nabla\mu
        +
        \Delta\mu \unknown\nabla\mu
        +
        (\nabla\mu\cdot\nabla \unknown)\nabla\mu
        \right).
        \end{split}
    \end{equation}
    Therefore,  integrating on $\Omega$ of both sides of \eqref{eq:2.Hesse-Diffusion},  we have
    \begin{equation}
        \begin{split}
        &\quad
         \int_{\Omega}(\nabla\mu\cdot\nabla^2(d(x)\unknown^{\alpha-1})\nabla\mu)\unknown\, dx
        \\
        &=   
        -
        \int_\Omega
        \nabla(d(x)\unknown^{\alpha-1})\cdot
        \left(
        \unknown\nabla^2\mu\nabla\mu
        +
        \Delta\mu \unknown\nabla\mu
        +
        (\nabla\mu\cdot\nabla \unknown)\nabla\mu
        \right)
        \, dx, 
        \end{split}
    \end{equation}    
    since the integral of the first term in the right-hand side of \eqref{eq:2.Hesse-Diffusion} vanishes by using the boundary condition of \eqref{Nonlinear-Fokker-Planck} with the divergence theorem. By direct computation of $\nabla(d(x)\unk^{\alpha-1})$,  we obtain \eqref{nabla^2(d(x)f^(alpha-1))1}.
\end{proof}

Plugging the above computation into \eqref{d^2F/dt^2_3},  the second time derivative of $\F[f](t)$  can be expressed as follows.
    \begin{equation}
    \label{d^2F/dt^2_4}
        \begin{split}
            \frac{d^2}{dt^2}\F[\unknown](t)&=2\int_\Omega (\nabla\mu\cdot\nabla^2\phi(x)\nabla\mu)\unknown\, dx
            \\
            &\quad
            -2\alpha(\alpha-1)\int_\Omega d(x)(\nabla \unknown\cdot\nabla^2\mu\nabla\mu)\unknown^{\alpha-1}\, dx\\
            &\quad+2\alpha(\alpha-1)\int_\Omega d(x)(\nabla \unknown\cdot\nabla\mu)\Delta\mu \unknown^{\alpha-1}\, dx
            \\
            &\quad
            +2\alpha(\alpha-1)\int_\Omega d(x)(\Delta\mu)^2\unknown^\alpha\, dx\\
            &\quad-2\alpha\int_\Omega (\nabla d(x)\cdot\nabla^2\mu\nabla\mu)\unknown^\alpha\, dx
            \\
            &\quad
            -2\alpha\int_\Omega  (\nabla d(x)\cdot\nabla\mu)\Delta\mu \unknown^\alpha\, dx            
            \\
            &\quad
            -2\alpha\int_\Omega (\nabla d(x)\cdot\nabla\mu)(\nabla \unknown\cdot\nabla\mu)\unknown^{\alpha-1}\, dx.
        \end{split}
    \end{equation}

If $d$ is a constant,  \eqref{d^2F/dt^2_4} coincides with the previous
result about the entropy dissipation methods by
\cites{MR1853037, MR3497125},  that is, the last three terms of the
right-hand side in \eqref{d^2F/dt^2_4} appear in the effect of
inhomogeneity of the diffusion.

We proceed with the computation according to the entropy dissipation methods.
We consider the third term in the right-hand side of
\eqref{d^2F/dt^2_4}.
 
\begin{lem}
    \label{nablafnablamuDeltamu}
    Let $\unknown$ be a classical solution of \eqref{Nonlinear-Fokker-Planck} on $\overline{\Omega}\times [0, \infty)$. Then
    \begin{equation}
    \label{D(nabla f nablamu)Deltamu f^alpha-1}
        \begin{split}
            &\quad
            2\alpha\int_\Omega d(x)(\nabla \unknown\cdot\nabla\mu)\Delta\mu \unknown^{\alpha-1}\, dx
            \\
            &=
            -
            2\int_\Omega d(x)(\Delta\mu)^2\unknown^\alpha\, dx
            -
            \int_\Omega d(x)\Delta |\nabla\mu|^2\unknown^\alpha\, dx
            \\
            &\quad
            +
            2\int_\Omega d(x) |\nabla^2\mu|^2\unknown^\alpha\, dx
            -
            2\int_\Omega (\nabla d(x)\cdot \nabla\mu)\Delta\mu \unknown^\alpha\, dx.
        \end{split}
    \end{equation}
\end{lem}
\begin{proof}
    First,  note that
    \begin{equation}
        \label{eq:2.Pohozaev-method-3}
        \alpha(d(x)\nabla\unk)\unk^{\alpha-1}
        =
        \nabla (d(x)\unknown^\alpha)
        -
        \unk^{\alpha}\nabla d(x).
    \end{equation}
    Next,  we compute $(\nabla(d(x)\unknown^\alpha)\cdot\nabla\mu)\Delta\mu$. 
    Writing a vector in component form,  we obtain
    \begin{equation}
        \label{eq:2.Pohozaev-method-1}
        (\nabla(d(x)\unknown^\alpha)\cdot\nabla\mu)\Delta\mu
        =\sum_{i, j=1}^n(d(x)\unknown^\alpha)_{x_i}\mu_{x_i}\mu_{x_jx_j}.
    \end{equation}
    Making a divergence form in the right-hand side of \eqref{eq:2.Pohozaev-method-1} as follows:
    \begin{equation}
        \label{eq:2.Pohozaev-method-2}
        (d(x)\unknown^\alpha)_{x_i}\mu_{x_i}\mu_{x_jx_j}
        =
        (d(x)\unknown^\alpha\mu_{x_i}\mu_{x_jx_j})_{x_i}
        -
        d(x)(\mu_{x_i}\mu_{x_jx_j})_{x_i}\unknown^\alpha.
    \end{equation}
    Compute the second term of the right-hand side of \eqref{eq:2.Pohozaev-method-2}
    as
    \begin{equation}
        (\mu_{x_i}\mu_{x_jx_j})_{x_i}
        =
        \mu_{x_ix_i}\mu_{x_jx_j}
        +
        \mu_{x_i}\mu_{x_jx_jx_i}
        =
        \mu_{x_ix_i}\mu_{x_jx_j}
        +
        (\mu_{x_i}\mu_{x_ix_j})_{x_j}
        -
        \mu_{x_ix_j}\mu_{x_ix_j}.
    \end{equation}
    Note that $\mu_{x_i}\mu_{x_ix_j}=\frac{1}{2}(\mu_{x_i}^2)_{x_j}$. Thus,  we arrive at
    \begin{equation}
    \label{eq:nabla(d unknown^alpha)}
        \begin{split}
        &\quad 
        (\nabla(d(x)\unknown^\alpha)\cdot\nabla\mu)\Delta\mu
        \\
        &=
        \sum_{i, j=1}^n
        \biggl(
         (d(x)\unknown^\alpha\mu_{x_i}\mu_{x_jx_j})_{x_i}
         -d(x)\mu_{x_ix_i}\mu_{x_jx_j}\unknown^\alpha
         \\
         &\qquad
         -\frac{d(x)}{2}(\mu_{x_i}^2)_{x_jx_j}\unknown^\alpha
         +d(x)\mu_{x_ix_j}\mu_{x_ix_j}\unknown^\alpha
          \biggr)
          \\
        &=
        \Div(d(x)\unknown^\alpha \nabla\mu\Delta\mu)
        -d(x)(\Delta\mu)^2\unknown^\alpha
        \\
        &\quad
        -\frac{1}{2}d(x)\Delta(|\nabla\mu|^2)\unknown^\alpha
        +d(x)|\nabla^2\mu|^2\unknown^\alpha. 
        \end{split}
    \end{equation}
    Therefore,  integrating on $\Omega$ of both side of \eqref{eq:nabla(d unknown^alpha)},  we have, 
    \begin{equation*}
        \begin{split}
            2\int_\Omega(\nabla (d(x)\unknown^\alpha)\cdot\nabla\mu)\Delta\mu\, dx
            &=-2\int_\Omega d(x)(\Delta\mu)^2\unknown^\alpha\, dx\\
            &\quad-\int_\Omega d(x)\Div(\nabla(|\nabla\mu|^2))\unknown^\alpha\, dx\\
            &\quad+2\int_\Omega d(x)|\nabla^2\mu|^2\unknown^\alpha\, dx, 
        \end{split}
    \end{equation*}
    since the integral of the first term in the right-hand side of \eqref{eq:nabla(d unknown^alpha)} vanishes by using the boundary condition of \eqref{Nonlinear-Fokker-Planck} with the divergence theorem. Using \eqref{eq:2.Pohozaev-method-3}, we obtain \eqref{D(nabla f nablamu)Deltamu f^alpha-1}.
\end{proof}

We next calculate the second term on the right-hand side of \eqref{d^2F/dt^2_4}.

\begin{lem}
    \label{nabla^2mu's-term}
     Let $\unknown$ be a classical solution of \eqref{Nonlinear-Fokker-Planck} on $\overline{\Omega}\times [0, \infty)$. Then
    \begin{equation}
    \label{nabla^2mu's-term_1}
        \begin{split}
            -2\alpha\int_\Omega d(x)(\nabla \unknown\cdot \nabla^2\mu\nabla\mu)\unknown^{\alpha-1}\, dx
            &
            =2\int_\Omega (\nabla d(x)\cdot\nabla^2\mu\nabla\mu)\unknown^\alpha\, dx
            \\
            &\quad
            -\int_{\partial\Omega} d(x)\unknown^\alpha\nabla (|\nabla\mu|^2)\cdot \nu\, d\sigma
            \\
            &\quad
            +\int_\Omega d(x) \Delta( |\nabla\mu|^2)\unknown^\alpha\, dx.
        \end{split}
    \end{equation}
    
\end{lem}

\begin{proof}
    Taking the inner product of $\nabla^2\mu\nabla\mu$ both side of \eqref{eq:2.Pohozaev-method-3},  we have
    \begin{equation}
    \label{eq:3.pohozaev-method_2.0}
        \alpha d(x)(\nabla \unk \cdot \nabla^2\mu\nabla\mu)\unknown^{\alpha-1}=(\nabla(d(x)\unknown^\alpha)\cdot\nabla^2\mu\nabla\mu)-(\nabla d(x)\cdot \nabla^2\mu\nabla\mu)\unk^{\alpha}.
    \end{equation}
    Next,  we compute $(\nabla(d(x)\unk^\alpha)\cdot\nabla^2\mu\nabla\mu)$. Writing a vector in component form,  we obtain
    \begin{equation}
    \label{eq:3.Pohozaev-method_2.1}
    (\nabla(d(x)\unk^\alpha)\cdot\nabla^2\mu\nabla\mu)
    =\sum_{i, j}^n(d(x)\unknown^\alpha)_{x_i}\mu_{x_ix_j}\mu_{x_j}.
    \end{equation}
    Making a divergence form in the right-hand side of \eqref{eq:3.Pohozaev-method_2.1} as follows: 
    \begin{equation}
    \label{eq:3.pohozaev-method_2.2}
        (d(x)\unknown^\alpha)_{x_i}\mu_{x_ix_j}\mu_{x_j}=(d(x)\unk^\alpha\mu_{x_ix_j}\mu_{x_j})_{x_i}-d(x)\unk^\alpha(\mu_{x_ix_j}\mu_{x_j})_{x_i}.
    \end{equation}
   Note that $\mu_{x_ix_j}\mu_{x_j}=\frac{1}{2}((\mu_{x_j})^2)_{x_i}$. Thus,  we arrive at 
   \begin{equation}
   \label{eq:3.pohozaev-method_2.3}
    \begin{split}
        (\nabla(d(x)\unk^\alpha)\cdot\nabla^2\mu\nabla\mu)
        &
        =\sum_{i, j=1}^n\left(\frac{1}{2}(d(x)\unk^\alpha(\mu_{x_j})^2_{x_i})_{x_i}-\frac{d(x)}{2}(\mu_{x_j})^2_{x_ix_i}\unknown^\alpha\right)
        \\
        &
        =\frac{1}{2}\Div(d(x)\unknown^\alpha\nabla|\nabla\mu|^2)-\frac{d(x)}{2}\Delta(|\nabla\mu|^2)\unknown^\alpha
    \end{split}
   \end{equation}
   Therefore,  integrating on $\Omega$ of both side of \eqref{eq:3.pohozaev-method_2.3},  we have, 
   \begin{equation}
        \begin{split}
            2\int_\Omega (\nabla(d(x)\unknown^\alpha)\cdot \nabla^2\mu\nabla\mu)\, dx
            &=
            \int_\Omega \Div(d(x)\unknown^\alpha\nabla|\nabla\mu|^2)\, dx
            \\
            &\quad
            -\int_\Omega d(x) \Delta(|\nabla\mu|^2)\unk^\alpha\, dx.
        \end{split}
   \end{equation}
   
    Using \eqref{eq:3.pohozaev-method_2.0} together with the divergence theorem,  we obtain \eqref{nabla^2mu's-term_1}. 
\end{proof}

Plugging \eqref{D(nabla f nablamu)Deltamu f^alpha-1} and \eqref{nabla^2mu's-term_1} into \eqref{d^2F/dt^2_4},  we obtain 
\begin{equation}
    \begin{split}
        &\quad
        \frac{d^2}{dt^2}\F[\unknown](t)
        \\
        &
        =2\int_\Omega (\nabla\mu\cdot\nabla^2\phi(x)\nabla\mu)\unknown\, dx
        +2(\alpha-1)\int_\Omega d(x) |\nabla^2\mu|^2\unknown^\alpha\, dx
        \\
        &\quad
        +2(\alpha-1)^2\int_\Omega d(x)(\Delta\mu)^2\unknown^\alpha\, dx-(\alpha-1)\int_{\partial\Omega} d(x)\unknown^\alpha\nabla (|\nabla\mu|^2)\cdot\nu\, d\sigma 
        \\
        &\quad
        -2\int_\Omega (\nabla d(x)\cdot\nabla^2\mu\nabla\mu)\unknown^\alpha\, dx-2(2\alpha-1)\int_\Omega  (\nabla d(x)\cdot\nabla\mu)\Delta\mu \unknown^\alpha\, dx 
        \\
        &\quad
        -2\alpha\int_\Omega (\nabla d(x)\cdot\nabla\mu)(\nabla \unknown\cdot\nabla\mu)\unknown^{\alpha-1}\, dx
        \\
        &=:
        2I_1+2(\alpha-1)I_2+2(\alpha-1)^2 I_3-(\alpha-1)I_4
        \\
        &\quad
        -2I_5-2(2\alpha-1)I_6-2\alpha I_7.
    \end{split}
\end{equation}

Since $\unk$ is positive,  we have $\nabla\mu\cdot\nu=0$ on $\partial\Omega$. Then,  it is well-known that the outer normal derivative of $|\nabla\mu|^2$ can be written as
\begin{equation}
    \nabla|\nabla\mu|^2\cdot\nu
    =
    2B_x(\nabla\mu, \nabla\mu), 
\end{equation}
at $x\in\partial\Omega$,  where $B_x$ is the second fundamental form at $x\in \partial \Omega$ (cf. \cite{MR555661}*{Lemma 5.3}, \cite{MR3348119}*{Lemma 4.2}). From the convexity assumption of $\Omega$,  the principal curvature of $\partial\Omega$ is non-positive thus we have
$I_4\leq0$. Therefore,  we obtain
\begin{equation}
    \label{second-derivative-F}
    \frac{d^2}{dt^2}\F[\unknown](t)\ge 2I_1+2(\alpha-1)I_2+2(\alpha-1)^2I_3-2I_5-2(2\alpha-1)I_6-2\alpha I_7.
\end{equation}

\begin{remark}
    If $d=1$,  then $\nabla d=0$ so \eqref{second-derivative-F} can be written as
    \begin{equation}
        \frac{d^2}{dt^2}\F[\unknown](t)\ge 2I_1+2(\alpha-1)I_2+2(\alpha-1)^2I_3, 
    \end{equation}
    which was deduced by \cite{MR1853037}. The above computation is based on     \cite{MR3497125}*{\S 2.5}. Inequality \eqref{second-derivative-F} is an extension of the previous result for the case where $d$ is not constant.
\end{remark}

To handle terms $I_5$,  $I_6$,  and $I_7$,  we prepare the following lemma. First,  we provide an estimate for $I_5$.

\begin{lem}
\label{lem:I5}
    Let $\unknown$ be a bounded, positive classical solution of \eqref{Nonlinear-Fokker-Planck} on $\overline{\Omega}\times [0, \infty)$. Then, 
    \begin{equation}
    \label{eq:2.Est_Grad_Diffusion_1}
        \begin{split}
            &\quad
            \left|\int_\Omega (\nabla d(x)\cdot \nabla^2\mu\nabla\mu)\unknown^\alpha\, dx\right|
            \\
            &\le \frac{\|\nabla d\|_\infty^2\|\unknown^{\alpha-1}\|_\infty}{2(\alpha-1)\min_{x \in \Omega } d(x)}\int_\Omega |\nabla\mu|^2\unknown\, dx
            +
            \frac{\alpha-1}{2}\int_{\Omega} d(x)|\nabla^2\mu|^2\unknown^\alpha\, dx.
        \end{split}
    \end{equation}
\end{lem}

\begin{proof}
    From the triangle inequality for integrals,  we have
    \begin{equation*}
            \left|\int_\Omega (\nabla d(x)\cdot\nabla^2\mu\nabla\mu)\unknown^\alpha\, dx\right|
            \le
            \int_\Omega  |\nabla d(x)||\nabla^2\mu||\nabla\mu|\unknown^\alpha\, dx.
    \end{equation*}
    Since $d(x)>0$,  it follows by H\"{o}lder's inequality and Young's inequality that
    \begin{equation*}
        \begin{split}
        &\quad
           \int_\Omega |\nabla d(x)||\nabla^2\mu||\nabla\mu|\unknown^\alpha \, dx
           \\
           &\le 
           \left(\int_\Omega \frac{1}{d(x)}|\nabla d(x)|^2|\nabla\mu|^2\unknown^\alpha\, dx\right)^{\frac{1}{2}}\left(\int_\Omega
           d(x)|\nabla^2\mu|^2\unknown^\alpha\, dx\right)^{\frac{1}{2}}
           \\
           &\le
           \frac{1}{2(\alpha-1)}\int_\Omega \frac{1}{d(x)}|\nabla d(x)|^2|\nabla\mu|^2\unknown^\alpha\, dx
            \\
            &\quad
            +\frac{(\alpha-1)}{2}\int_\Omega d(x) |\nabla^2\mu|^2\unknown^\alpha\, dx.
        \end{split}
    \end{equation*}
    Using the boundedness of $\nabla d(x)$ and $\unknown$,  we have
    \begin{equation*}
        \int_\Omega \frac{1}{d(x)}|\nabla d(x)|^2|\nabla\mu|^2\unknown^\alpha\, dx
        \le
        \frac{\|\nabla d\|_\infty^2\|\unknown^{\alpha-1}\|_\infty}{\min_{x \in \Omega} d(x)}\int_\Omega |\nabla\mu|^2\unknown\, dx.
    \end{equation*}
    Summarizing the above,  we obtain \eqref{eq:2.Est_Grad_Diffusion_1}.
\end{proof}

From \eqref{eq:2.Est_Grad_Diffusion_1},  we obtain
\begin{equation}
    \label{eq:2.Estimate_I5}
    2|I_5|
    \le 
    \frac{\|\nabla d\|_\infty^2\|\unknown^{\alpha-1}\|_\infty}{(\alpha-1)\min_{x \in \Omega } d(x)}
    \D[\unk](t)
    +
    (\alpha-1)I_2.
\end{equation}

Next, we estimate $I_6$ by $\D[\unk]$ and $I_3$.

\begin{lem}
    \label{lem:I6}
    Let $\unknown$ be a bounded, positive classical solution of \eqref{Nonlinear-Fokker-Planck} on $\overline{\Omega}\times [0, \infty)$. Then
    \begin{equation}
        \label{eq:2.Est_Grad_Diffusion_2}
        \begin{split}
        \left|\int_\Omega  (\nabla d(x)\cdot\nabla\mu)\Delta\mu \unknown^\alpha\, dx\right|
        &\le
        \frac{(2\alpha-1)\|\nabla d\|_\infty^2 \|\unknown^{\alpha-1}\|_\infty}{4(\alpha-1)^2\min_{x\in\Omega} d(x)}\int_\Omega |\nabla\mu|^2\unknown\, dx
        \\ 
        &\quad
        +\frac{(\alpha-1)^2}{2\alpha-1}\int_\Omega d(x)(\Delta\mu)^2\unknown^\alpha \, dx.
        \end{split}
    \end{equation}
\end{lem}

\begin{proof}
    From the triangle inequality for integrals,  we have
    \begin{equation*}
        \left|\int_\Omega (\nabla d(x)\cdot\nabla\mu)\Delta\mu \unknown^\alpha\, dx\right|
        \le
        \int_\Omega |\nabla d(x)||\nabla\mu||\Delta\mu|\unknown^\alpha\, dx .
    \end{equation*}
    Similarly, in the proof of Lemma \ref{lem:I5},  it follows from H\"older's and Young's inequality that
    \begin{equation*}
        \begin{split}
            &\quad
            \int_\Omega |\nabla d(x)||\nabla\mu||\Delta\mu|\unknown^\alpha\, dx 
            \\
            &\le 
            \left(\int_\Omega \frac{|\nabla d(x)|^2}{d(x)}|\nabla\mu|^2\unknown^\alpha\, dx\right)^{\frac{1}{2}}\left(\int_\Omega d(x)(\Delta\mu)^2\unknown^\alpha \, dx\right)^{\frac{1}{2}} 
            \\
            &\le 
            \frac{2\alpha-1}{4(\alpha-1)^2}\int_\Omega \frac{|\nabla d(x)|^2}{d(x)}|\nabla\mu|^2\unknown^\alpha\, dx
            \\
            &\quad
            +\frac{(\alpha-1)^2}{2\alpha-1}\int_\Omega d(x)(\Delta\mu)^2\unknown^\alpha \, dx.
        \end{split}
    \end{equation*}
    As $\nabla d(x)$ and $\unknown$ are bounded,  we have
    \begin{equation*}
        \int_\Omega \frac{|\nabla d(x)|^2}{d(x)}|\nabla\mu|^2\unknown^\alpha\, dx
        \le\frac{\|\nabla d\|_\infty^2 \|\unknown^{\alpha-1}\|_\infty}{\min_{x\in\Omega} d(x)}\int_\Omega |\nabla\mu|^2\unknown\, dx.
    \end{equation*}
    Therefore,  \eqref{eq:2.Est_Grad_Diffusion_2} follows from summarizing the above estimates.
\end{proof}

From \eqref{eq:2.Est_Grad_Diffusion_2},  we obtain
\begin{equation}
    \label{eq:2.Estimate_I6}
    2(2\alpha-1)|I_6|
    \leq    
    \frac{(2\alpha-1)^2\|\nabla d\|_\infty^2 \|\unknown^{\alpha-1}\|_\infty}{2(\alpha-1)^2\min_{x\in\Omega} d(x)}\D[\unk](t)
    +
    2(\alpha-1)^2I_3
\end{equation}

To proceed to estimate $I_7$,  we first substitute $\nabla\unk$ by $\nabla\mu$. In the next lemma,  we use the relation $\alpha d(x)\unknown^{\alpha-1}+\phi$.

\begin{lem}
    \label{lem:I7}
    Let $\unknown$ be a classical solution of \eqref{Nonlinear-Fokker-Planck} on $\overline{\Omega}\times [0, \infty)$. Then
    \begin{equation}
    \label{eq:2.Compute_Grad_Diffusion_3}
        \begin{split}
            &\quad
            \alpha \int_\Omega (\nabla d(x)\cdot \nabla\mu)(\nabla \unknown\cdot \nabla\mu)\unknown^{\alpha-1}\, dx
            \\
            &=
            \frac{1}{\alpha-1}\int_\Omega\frac{1}{d(x)}(\nabla d(x)\cdot \nabla\mu)|\nabla\mu|^2\unknown\, dx
            \\
            &\quad
            -\frac{\alpha}{\alpha-1}\int_\Omega \frac{1}{d(x)}(\nabla d(x)\cdot \nabla\mu)^2\unknown^\alpha\, dx
            \\
            &\quad
            -\frac{1}{\alpha-1}\int_\Omega \frac{1}{d(x)}(\nabla d(x)\cdot\nabla\mu)(\nabla\phi(x)\cdot\nabla\mu)\unknown\, dx.
        \end{split}
    \end{equation}
\end{lem}

\begin{proof}
    First note that $(\alpha-1)\unk^{\alpha-1}\nabla\unk = \unk\nabla\unk^{\alpha-1}$. Taking the gradient of both side of $\mu=\alpha d(x)\unknown^{\alpha-1}+\phi(x)$,  we have
    \begin{equation}
        \nabla\mu
        =
        \alpha d(x)\nabla\unknown^{\alpha-1}+\alpha \unknown^{\alpha-1}\nabla d(x)+\nabla\phi(x).
    \end{equation}
    Thus,  the integrand of $I_7$ turns into
    \begin{equation}
        \begin{split}
        &\quad
       (\nabla d(x)\cdot \nabla\mu)(\nabla \unknown\cdot \nabla\mu)\unknown^{\alpha-1}
        \\
        &=
        \frac{1}{\alpha-1}(\nabla d(x)\cdot \nabla\mu)(\nabla \unknown^{\alpha-1}\cdot \nabla\mu)\unknown
        \\
        &=
        \frac{1}{\alpha(\alpha-1) d(x)}
        \left(
        (\nabla d(x)\cdot \nabla\mu)
        \left(
        \left(
        \nabla\mu
        -
        \alpha\unk^{\alpha-1}\nabla d(x)
        -
        \nabla\phi(x)
        \right)
        \cdot \nabla\mu\right)
        \right)
        \unknown
        \end{split}
     \end{equation}
    Taking the integration on $\Omega$ on both sides, we obtain \eqref{eq:2.Compute_Grad_Diffusion_3}.
\end{proof}

Note that the second term of the right-hand side of \eqref{eq:2.Compute_Grad_Diffusion_3} is non-positive,  one have from \eqref{eq:2.Compute_Grad_Diffusion_3} that
\begin{equation}
    \label{eq:2.estimate_I7}
    \begin{split}
        -2\alpha I_7
        &\geq
        -\frac{2}{\alpha-1}\int_\Omega\frac{1}{d(x)}(\nabla d(x)\cdot \nabla\mu)|\nabla\mu|^2\unknown\, dx
        \\
        &\quad
        +\frac{2}{\alpha-1}\int_\Omega \frac{1}{d(x)}(\nabla d(x)\cdot\nabla\mu)(\nabla\phi(x)\cdot\nabla\mu)\unknown\, dx
        \\
        &\ge-\frac{2\|\nabla d\|_\infty}{(\alpha-1)\min_{x \in \Omega}d(x)}\int_\Omega |\nabla\mu|^3\unknown\, dx
        \\
        &\quad
        -\frac{2\|\nabla d\|_\infty \|\nabla\phi\|_\infty}{(\alpha-1)\min_{x \in \Omega}d(x)}\D[\unk](t).    
    \end{split}
\end{equation}

We need to handle a cubic nonlinearity in the right-hand side of \eqref{eq:2.estimate_I7}. Since $\unknown$ is bounded and strictly positive,  we can use the following Sobolev-Poincar\'e type inequality.
\begin{prop}%[Sobolev-Poincare type inequality]
\label{Sobolev}
    Let $\unknown$ be a bounded,  strictly positive classical solution of 
    \eqref{Nonlinear-Fokker-Planck} on $\overline{\Omega}\times [0, \infty)$. Then,  there is a suitable positive constant $\Cl{const:sobolevtype}>0$ depending only of $n$ and $\Omega$ such that for any vector field $\bv\in C^1(\Omega)$, 
    \begin{equation}
    \label{Sobolev-Poincare}
        \left(\int_\Omega |\bv-\overline{\bv}|^{p*}\unknown\, dx\right)^{\frac{1}{p*}}\le \Cr{const:sobolevtype}\left(\int_\Omega |\nabla\bv|^2\unknown\, dx\right)^{\frac{1}{2}}, 
    \end{equation}
    where $\overline{\bv}$ is the integral average of $\bv$ and the $p*$ is an exponent satisfying $\frac{1}{p*}=\frac{1}{2}-\frac{1}{n}$ for $n\ge3$ and arbitrary $2 \le p* < \infty$ for $n=1, 2$.
\end{prop}
\begin{proof}
    Since $p*$ is the Sobolev exponent,  it follows from the Sobolev-Poincar\'e inequality (\cite{MR1814364}*{p.174}, \cite{MR1817225}*{Theorem4.3}) that
    \begin{equation*}
        \left(\int_\Omega |\bv-\overline{\bv}|^{p*}\, dx\right)^{\frac{1}{p*}}\le \Cl{const:Sobolev}\left(\int_\Omega |\nabla\bv|^2\, dx\right)^{\frac{1}{2}}
    \end{equation*}
    holds for any vector field $\bv \in C^1(\Omega)$. By the definition of $\Cr{const:f>c}$,  $\Cr{const:f<c}$,  we obtain that
    \begin{equation*}
        \begin{split}
            \left(\int_\Omega |\bv-\overline{\bv}|^{p*}\unknown\, dx\right)^{\frac{1}{p*}}
            &\le
            \Cr{const:f<c}^{\frac{1}{p*}}\left(\int_\Omega |\bv-\overline{\bv}|^{p*}\, dx\right)^{\frac{1}{p*}}\\
            &\le 
            \Cr{const:f<c}^{\frac{1}{p*}}\Cr{const:Sobolev}\left(\int_\Omega |\nabla\bv|^2\, dx\right)^{\frac{1}{2}}
            \le 
            \frac{\Cr{const:f<c}^{\frac{1}{p*}}\Cr{const:Sobolev}}{\Cr{const:f>c}^{\frac{1}{2}}}\left(\int_\Omega |\nabla\bv|^2\unknown\, dx\right)^{\frac{1}{2}}.
        \end{split}
    \end{equation*}
\end{proof}

We prove an interpolation inequality from the Sobolev-Poincar\'e type inequality \eqref{Sobolev-Poincare}.
\begin{prop}
\label{prop:|v|^3}
    Let $n=1, 2, 3$. Let $\unknown$ be a bounded, strictly positive solution of \eqref{Nonlinear-Fokker-Planck} on $\overline{\Omega}\times [0, \infty)$. Then,  there are constants $\Cl{const:|v|^3->|nablav|^2}$,  $\Cl{const:|v|^3-|v|^2, 1}$,  and $\Cl{const:|v|^3-|v|^2, 2}>0$ such that for any $\bv\in C^1(\Omega)$, 
    \begin{equation}
        \label{eq:2.Interpolation}
            \int_\Omega |\bv|^3\unknown\, dx
	    \le
	    \Cr{const:|v|^3->|nablav|^2}
	    \int_\Omega |\nabla\bv|^2\unknown\, dx
	    +
	    \Cr{const:|v|^3-|v|^2, 1}
	    \left(\int_\Omega |\bv|^2\unknown\, dx\right)^3
            +
	    \Cr{const:|v|^3-|v|^2, 2}
	    \left(\int_\Omega |\bv|^2\unknown\, dx\right)^{\frac{3}{2}}.
    \end{equation}
\end{prop}

\begin{remark}
 \label{rem:2.dependency_constant_Interpolation} Note that the constants
 $\Cr{const:|v|^3->|nablav|^2}$,  $\Cr{const:|v|^3-|v|^2, 1}$,  and
 $\Cr{const:|v|^3-|v|^2, 2}$ are independent of $\Cr{const:D_min}$,  the
 lower bounds of $d$. These constants depend on
 $\Cr{const:f>c}$ and $\Cr{const:f<c}$,  the lower bounds and the upper
 bounds of $\unk$, nevertheless solution $\unk$ of
 \eqref{Nonlinear-Fokker-Planck} may depend on the diffusion coefficient
 $d$. Here,  we regard $\Cr{const:f>c}$ and $\Cr{const:f<c}$ independent
 of $d$. We will comment on this relation later.
\end{remark}

\begin{proof}
    Let $a, b>0$ such that $a+b=1$,  and let $p>1$ satisftying $3ap\geq1$. Then by H\"{o}lder's and convex inequality, 
    \begin{equation}
    \label{|v|^3f'sholder}
        \begin{split}
            \int_\Omega |\bv|^3\unknown\, dx&\le \left(\int_\Omega |\bv|^{3ap}\unknown\, dx \right)^{\frac{1}{p}}\left(\int_\Omega |\bv|^{3bp'}\unknown\, dx\right)^{\frac{1}{p'}} \\
            &=\left(\int_\Omega |\bv-\overline{\bv}+\overline{\bv}|^{3ap}\unknown\, dx \right)^{\frac{1}{p}}\left(\int_\Omega |\bv|^{3bp'}\unknown\, dx\right)^{\frac{1}{p'}} \\
            &\le 2^{\frac{3ap-1}{p}}\Biggl(\left(\int_\Omega |\bv-\overline{\bv}|^{3ap}\unknown\, dx \right)^{\frac{1}{p}}
            \\
            &\quad
            +\left(\int_\Omega |\overline{\bv}|^{3ap}\unknown\, dx \right)^{\frac{1}{p}}\Biggr)\left(\int_\Omega |\bv|^{3bp'}\unknown\, dx\right)^{\frac{1}{p'}}, 
        \end{split}
    \end{equation}
    where $p'$ is the H\"{o}lder dual index of $p$. Next, we set $3ap=p*\geq1$. By Proposition \ref{Sobolev}, 
    \begin{equation*}
        \left(\int_\Omega |\bv-\overline{\bv}|^{3ap}\unknown\, dx \right)^{\frac{1}{p}}
        \le 
        \Cr{const:sobolevtype}^{\frac{p*}{p}}\left(\int_\Omega |\nabla\bv|^{2}\unknown\, dx \right)^{\frac{p*}{2p}}.
    \end{equation*}
    We take $3bp'=2$ and $\frac{p*}{2p}<1$. Then,  by using Young's inequality, 
    \begin{equation*}
    \begin{split}
        \left(\int_\Omega |\nabla\bv|^{2}\unknown\, dx \right)^{\frac{p*}{2p}}\left(\int_\Omega |\bv|^{2}\unknown\, dx\right)^{\frac{1}{p'}}
            &
            \le
            \frac{p*}{2p}\int_\Omega |\nabla\bv|^2\unknown\, dx
            \\
            &\quad
            +\left(1-\frac{p*}{2p}\right)\left(\int_\Omega |\bv|^2\unknown\, dx\right)^{\frac{1}{p'}\left(1-\frac{p*}{2p}\right)^{-1}}.
    \end{split}
    \end{equation*}
    Note from \eqref{eq;apripri} that $\Cr{const:f>c}$ is the minimum of $\unknown$ on $\overline{\Omega} \times [0, \infty)$. Then,  from H\"{o}lder's inequality and \eqref{eq; f_is_PDF} that
    \begin{equation*}
        \begin{split}
            \left(\int_\Omega |\overline{\bv}|^{p*}\unknown\, dx \right)^{\frac{1}{p}}
            =|\overline{\bv}|^{\frac{p*}{p}}
            &\le
            \left(\frac{1}{|\Omega|}\int_\Omega |\bv|\, dx\right)^{\frac{p*}{p}}
            \\
            &\le
            \left(\frac{1}{|\Omega|\Cr{const:f>c}}\int_\Omega |\bv|\unknown\, dx\right)^{\frac{p*}{p}}
            \\
            &\le 
            \left(\frac{1}{|\Omega|\Cr{const:f>c}}\right)^{\frac{p*}{p}}\left(\int_\Omega |\bv|^2\unknown\, dx\right)^{\frac{p*}{2p}}.
        \end{split}
    \end{equation*}
    Therefore subsituting the above inequality to \eqref{|v|^3f'sholder},  we obtain
    \begin{equation*}
        \begin{split}
            \int_\Omega |\bv|^3\unknown\, dx
            &
            \le 2^{\frac{3ap-1}{p}}\Cr{const:sobolevtype}^{\frac{p*}{p}}\frac{p*}{2p}\int_\Omega |\nabla\bv|^2\unknown\, dx
            \\
            &\quad
            +2^{\frac{3ap-1}{p}}
            \Cr{const:sobolevtype}^{\frac{p*}{p}}\left(1-\frac{p*}{2p}\right)\left(\int_\Omega |\bv|^2\unknown\, dx\right)^{\frac{1}{p'}\left(1-\frac{p*}{2p}\right)^{-1}} 
            \\
            &\quad
            +2^{\frac{3ap-1}{p}}\left(\frac{1}{|\Omega|\Cr{const:f>c}}\right)^{\frac{p*}{p}}\left(\int_\Omega |\bv|^2\unknown\, dx\right)^{\frac{p*}{2p}+\frac{1}{p'}}.
        \end{split}
    \end{equation*}
    Next, we check the constraints' condition. If $n\ge3$,  note that $a+b=1$ and $p'$ is the H\"{o}lder dual index of $p$,  $3ap=p*$ and $p*$ is the Sobolev exponent. Then,  we get
    \begin{equation*}
            1=\frac{1}{p}+\frac{1}{p'}=\frac{3a}{p*}+\frac{3b}{2}=\frac{3}{2}-\frac{3a}{n}.
    \end{equation*}
    Thus,  we obtain $a=\frac{n}{6}$. Combining $\frac{P*}{2p}<1$ and $3ap=p*$,  we deduce $a<\frac{2}{3}$ hence $n<4$,  which means $n=3$. If $n=1, 2$,  we put $p*=6$. Then we deduce from $3ap=6$,  $3bp'=2$ that
    \begin{equation}
        1=\frac{1}{p}+\frac{1}{p'}=\frac{a}{2}+\frac{3}{2}=\frac{1}{2}+b, 
    \end{equation}
    thus $a=b=\frac{1}{2}$,  $p=4$,  $p'=\frac{4}{3}$,  and $\frac{1}{p'}(1-\frac{p*}{2p})^{-1}$.
    \begin{equation*}
        a=b=\frac{1}{2}, \ 3bp'=2\Leftrightarrow p'=\frac{4}{3}, \  p=4, \  p*=6, \  \frac{1}{p'}\left(1-\frac{p*}{2p}\right)^{-1}=3.
    \end{equation*}
    Using the above results,  we obtain \eqref{eq:2.Interpolation},  where
    \begin{equation*}
            \Cr{const:|v|^3->|nablav|^2}:=2^{-\frac{3}{4}}3\Cr{const:sobolevtype}^{\frac{3}{2}}, \quad
            \Cr{const:|v|^3-|v|^2, 1}:=2^{-\frac{3}{4}}\Cr{const:sobolevtype}^{\frac{3}{2}},  \quad
            \Cr{const:|v|^3-|v|^2, 2}:=2^{\frac{5}{4}}\left(\frac{1}{|\Omega|\Cr{const:f>c}}\right)^{\frac{3}{2}}.
    \end{equation*}
\end{proof}

Using Lemma \ref{lem:I5},  \ref{lem:I6},  \ref{lem:I7} and Proposition \ref{prop:|v|^3} to \eqref{second-derivative-F},  we obtain the following estimate:
\begin{lem}
\label{d^2F/dt^2}
 Let $n=1, 2, 3$.
    Let $\unknown$ be a bounded,  positive classical solution of \eqref{Nonlinear-Fokker-Planck} on $\overline{\Omega}\times [0, \infty)$. 
 Then,  there are constants $\Cl{const:f, nablaD, nablaphi}$, $\Cl{const:f, nablaD, nablaphi1}$, $\Cl{const:f, nablaD, nablaphi2}$ and $\Cl{const:f, nablaD, nablaphi3}>0$ depending only on $\|\nabla d\|_\infty$,  $\|\nabla \phi\|_\infty$,  $\| \unk\|_\infty$,  $n$,  $\alpha$,  and $\Omega$ such that, 
    \begin{equation}
    \label{eq:2.Second_time_derivative_F_SemiFinal}
    \begin{split}
    \frac{d^2}{dt^2}\F[\unknown](t)
    &
    \ge
    \left(    
    2\lambda
    -
    \frac{\Cr{const:f, nablaD, nablaphi}}{\Cr{const:D_min}}
    \right)
    \D[\unk](t)
    +
    \left((\alpha-1)-\frac{\Cr{const:f, nablaD, nablaphi1}}{\Cr{const:D_min}^2}\right)I_2
    \\
    &
    \quad
    -
    \frac{\Cr{const:f, nablaD, nablaphi2}}{\Cr{const:D_min}}
    \left(\D[\unk](t)\right)^3
    -
    \frac{\Cr{const:f, nablaD, nablaphi3}}{\Cr{const:D_min}}
    \left(\D[\unk](t)\right)^{\frac{3}{2}}
    \end{split}
    \end{equation}
\end{lem}

\begin{remark}
\label{rem3}
 Again,  we note as Remark \ref{rem:2.dependency_constant_Interpolation}
 that the constants $\Cr{const:f, nablaD, nablaphi}$, 
 $\Cr{const:f, nablaD, nablaphi1}$, $\Cr{const:f, nablaD, nablaphi2}$ and
 $\Cr{const:f, nablaD, nablaphi3}>0$ are not depending on
 $\Cr{const:D_min}$,  the lower bounds of $d$.  
\end{remark}

\begin{proof}
Plugging \eqref{eq:2.Estimate_I5},  \eqref{eq:2.Estimate_I6},  and \eqref{eq:2.estimate_I7} into \eqref{second-derivative-F} and using $\Cr{const:D_min}=\min_{x\in\Omega}d(x)$,  we can estimate the second time derivative
of $\F$ as
\begin{equation}
\label{Second-Order-Derivative_of_F}
    \frac{d^2}{dt^2}\F[\unknown](t)
    \ge
    2I_1
    +
    (\alpha-1)I_2
    -
    \frac{\Cr{const:f, nablaD, nablaphi}}{\Cr{const:D_min}}
    \D[\unk](t)
    -
    \frac{2\|\nabla d\|_\infty}{(\alpha-1)\Cr{const:D_min}}\int_\Omega |\nabla\mu|^3\unknown\, dx, 
\end{equation}
where
\begin{equation}
    \Cr{const:f, nablaD, nablaphi}
    :=
    \frac{(2\alpha-1)^2\|\nabla d\|_\infty^2\|\unknown^{\alpha-1}\|_\infty}{(\alpha-1)^2}
    +\frac{\|\nabla d\|_\infty^2\|\unknown^{\alpha-1}\|_\infty}{(\alpha-1)}
    +\frac{2\|\nabla d\|_\infty\|\nabla \phi\|_\infty}{(\alpha-1)}.
\end{equation}

From proposition \ref{prop:|v|^3} with $\bv=\nabla\mu$ and using $d\geq\Cr{const:D_min}$,  we have
\begin{equation}
\label{|nablamu|^3}
            \int_\Omega |\nabla \mu|^3\unknown\, dx
            \le 
            \frac{\Cr{const:|v|^3->|nablav|^2}}{\Cr{const:D_min}}I_2
            +
            \Cr{const:|v|^3-|v|^2, 1}\left(\D[\unk](t)\right)^3
            +\Cr{const:|v|^3-|v|^2, 2}\left(\D[\unk](t)\right)^{\frac{3}{2}}.
\end{equation}
Plugging \eqref{|nablamu|^3} into \eqref{Second-Order-Derivative_of_F},  we obtain
\begin{equation} 
    \label{eq:2.Second-Order-Derivative_of_F_2}
    \begin{split}
    \frac{d^2}{dt^2}\F[\unknown](t)
    &
    \ge
    2I_1
    +
    \left((\alpha-1)-\frac{\Cr{const:f, nablaD, nablaphi1}}{\Cr{const:D_min}^2}\right)I_2
    -
    \frac{\Cr{const:f, nablaD, nablaphi}}{\Cr{const:D_min}}
    \D[\unk](t)\\
    &\quad
    -\frac{\Cr{const:f, nablaD, nablaphi2}}{\Cr{const:D_min}}
    \left(\D[\unk](t)\right)^3
    -\frac{\Cr{const:f, nablaD, nablaphi3}}{\Cr{const:D_min}}
    \left(\D[\unk](t)\right)^{\frac{3}{2}}, 
    \end{split}
    \end{equation}
where
\begin{equation}
    \begin{split}
        \Cr{const:f, nablaD, nablaphi1}&:=\frac{2\|\nabla d\|_\infty}{(\alpha-1)}\Cr{const:|v|^3->|nablav|^2}, \quad
        \Cr{const:f, nablaD, nablaphi2}:=\frac{2\|\nabla d\|_\infty}{(\alpha-1)}\Cr{const:|v|^3-|v|^2, 1}, \quad
        \Cr{const:f, nablaD, nablaphi3}:=\frac{2\|\nabla d\|_\infty}{(\alpha-1)}\Cr{const:|v|^3-|v|^2, 2}.
    \end{split}
\end{equation}
Since $\nabla^2 \phi\ge \lambda I$,  $I_1$ can be estimated by
\begin{equation}
    \label{eq:2.Estimate.Hesse.potential}
    2I_1
    \ge
    2\lambda \int_\Omega |\nabla\mu|^2\unknown\, dx
    =
    2\lambda \D[\unk](t)
\end{equation}
Therefore plugging \eqref{eq:2.Estimate.Hesse.potential} into
\eqref{eq:2.Second-Order-Derivative_of_F_2},  we obtain \eqref{eq:2.Second_time_derivative_F_SemiFinal}.
\end{proof}

Now,  we take $\Cr{const:D_min}$ large enough to control the coefficient
of the first and second terms of
\eqref{eq:2.Second_time_derivative_F_SemiFinal}.

\begin{lem}
 \label{D>0} Let $n=1, 2, 3$.  Let $\unknown$ be a bounded,  positive
 classical solution of \eqref{Nonlinear-Fokker-Planck} on
 $\overline{\Omega}\times [0, \infty)$. Then,  there is a large enough number $\Cr{const:D_min}>0$ such that if $d(x) >\Cr{const:D_min}$ on $x
 \in \Omega$,  then there exists constants $\Cl{const:|nablamu|^2f^3/2}$, 
 $\Cl{const:|nablamu|^2f^3}>0$ depending only on $\|\nabla d\|_\infty$, 
 $\|\nabla \phi\|_\infty$,  $\| \unknown\|_\infty$,  $n$,  $\alpha$,  and
 $\Omega$ such that
 \begin{equation}
  \label{eq:2.Second_time_derivative_F_Final}
  \frac{d^2}{dt^2}\F[\unknown](t)
    \ge
    \lambda
    \D[\unk](t)
    -
    \Cr{const:|nablamu|^2f^3/2}
    \left(\D[\unk](t)\right)^3
    -
    \Cr{const:|nablamu|^2f^3}
    \left(\D[\unk](t)\right)^{\frac{3}{2}}
    \quad
    t>0.
 \end{equation}
\end{lem}

\begin{proof}
 Let $\Cr{const:D_min}$ be large enough such that
 \begin{equation*}
  2\lambda-\frac{\Cr{const:f, nablaD, nablaphi}}{\Cr{const:D_min}}
   \ge
   \lambda, 
   \quad
   (\alpha-1)-\frac{\Cr{const:f, nablaD, nablaphi1}}{\Cr{const:D_min}}
   \ge
   0.
 \end{equation*}
 Then,  the second time derivative of $\F$ can be estimated as
 \begin{equation*}
  \frac{d^2}{dt^2}\F[\unknown](t)
    \ge
    \lambda
    \D[\unk](t)
    -
    \frac{\Cr{const:f, nablaD, nablaphi2}}{\Cr{const:D_min}}
    \left(\D[\unk](t)\right)^3
    -
    \frac{\Cr{const:f, nablaD, nablaphi3}}{\Cr{const:D_min}}
    \left(\D[\unk](t)\right)^{\frac{3}{2}}
 \end{equation*}
 Thus,  we obtain \eqref{eq:2.Second_time_derivative_F_Final} by takind
 constants as 
 \begin{equation*}
  \Cr{const:|nablamu|^2f^3/2}
   :=
   \frac{\Cr{const:f, nablaD, nablaphi2}}{\Cr{const:D_min}}, 
   \quad
   \Cr{const:|nablamu|^2f^3}
   :=
   \frac{\Cr{const:f, nablaD, nablaphi3}}{\Cr{const:D_min}}.
 \end{equation*}
\end{proof}

From differential inequality
\eqref{eq:2.Second_time_derivative_F_Final},  we use the following
Gronwall type lemma.

\begin{lemma}
 \label{lem:2.Gronwall_Lemma}
 Let $g:[0, \infty)\rightarrow\R$ be a differentialble function. Assume
 there exist positive constants $\Cl{const:Gronwall1}, 
 \Cl{const:Gronwall2}$,  and $\Cl{const:Gronwall3}>0$ such that
 \begin{equation}
  \label{ineq:gronwall_of_g}
   \frac{d}{dt}g(t)
   \le
   -
   \Cr{const:Gronwall1} g(t)
   + 
   \Cr{const:Gronwall2}
   g(t)^{\frac{3}{2}}
   +
   \Cr{const:Gronwall3}
   g(t)^3
 \end{equation}
 for any $t>0$. Then,  there exist positive constants
 $\Cl{const:Gronwall_Initial},  \Cl{const:Gronwall_Coefficient}>0$
 depending only on $\Cr{const:Gronwall1},  \Cr{const:Gronwall2}$,  and
 $\Cr{const:Gronwall3}>0$ such that if
 $g(0)<\Cr{const:Gronwall_Initial}$,  then $g(t)\leq
 \Cr{const:Gronwall_Coefficient}e^{-\Cr{const:Gronwall1}t}$.
\end{lemma}

\begin{proof}
 Let $G(t):=e^{\Cr{const:Gronwall1} t}g(t)$ and we will show that
 $G(T)\leq\Cr{const:Gronwall_Coefficient}$ for all $T>0$ if $G(0)=g(0)<
 \Cr{const:Gronwall_Initial} $. From \eqref{ineq:gronwall_of_g},  we
 obtain
 \begin{equation*}
  \begin{split}
   \frac{dG}{dt}
   &\le
   \Cr{const:Gronwall2} e^{\Cr{const:Gronwall1} t}g(t)^{\frac{3}{2}}
   +
   \Cr{const:Gronwall3} e^{\Cr{const:Gronwall1} t} g(t)^3 
   \\
   &=
   \Cr{const:Gronwall2} e^{-\frac{1}{2}\Cr{const:Gronwall1} t}
   G(t)^{\frac{3}{2}}
   \Cr{const:Gronwall3} e^{-2\Cr{const:Gronwall1} t} G(t)^3 
   \\
   &\leq
   e^{-\frac{1}{2}\Cr{const:Gronwall1} t}
   \left(
   \Cr{const:Gronwall2} 
   G(t)^{\frac{3}{2}}
   +
   \Cr{const:Gronwall3} G(t)^3 
   \right).
  \end{split}
 \end{equation*} 
 Thus,  we have
 \begin{equation}
  \label{ineq:dG/dt}
   \frac{1}{
   \Cr{const:Gronwall2} 
   G(t)^{\frac{3}{2}}
   +
   \Cr{const:Gronwall3} G(t)^3 
   } 
   \frac{dG}{dt}
   \le
   e^{-\frac{1}{2}\Cr{const:Gronwall1} t}.
 \end{equation}
 Integrating the differential inequality \eqref{ineq:dG/dt} with respect to $t \in [0, T]$,  we obtain
 \begin{equation}
  \label{G's_integrate}
   \int_{G(0)}^{G(T)}
   \frac{1}{
   \Cr{const:Gronwall2} 
   \xi^{\frac{3}{2}}
   +
   \Cr{const:Gronwall3} \xi^3 
   } 
   \, d\xi
   \le
   \int_0^T e^{-\frac{1}{2}\Cr{const:Gronwall1} t}\, dt
   \leq
   \frac{2}{\Cr{const:Gronwall1}}.   
 \end{equation}
 We focus on the integral on the left-hand side of
 \eqref{G's_integrate}. Decomposing the integrand of the left-hand side
 of \eqref{G's_integrate},  we obtain
 \begin{equation*}
  \frac{1}{
   \Cr{const:Gronwall2} 
   \xi^{\frac{3}{2}}
   +
   \Cr{const:Gronwall3} \xi^3 
   } 
   =
   \frac{1}{
   \Cr{const:Gronwall2} 
   \xi^{\frac{3}{2}}
   } 
   -
   \frac{\Cr{const:Gronwall3}}{
   \Cr{const:Gronwall2}
   \left(
   \Cr{const:Gronwall2} 
   +
   \Cr{const:Gronwall3} \xi^{\frac{3}{2}} 
   \right)
   }. 
 \end{equation*}
 Thus,  we have
 \begin{equation}
  \label{eq:G_split}
  \begin{split}
      \int_{G(0)}^{G(T)}
   \frac{1}{
   \Cr{const:Gronwall2} 
   \xi^{\frac{3}{2}}
   +
   \Cr{const:Gronwall3} \xi^3 
   } 
   \, d\xi
   &
   =
   \int_{G(0)}^{G(T)}
   \frac{1}{
   \Cr{const:Gronwall2} 
   \xi^{\frac{3}{2}}
   }
   \, d\xi
   \\
   &\quad
   -
   \int_{G(0)}^{G(T)}
   \frac{\Cr{const:Gronwall3}}{
   \Cr{const:Gronwall2}
   \left(
    \Cr{const:Gronwall2} 
    +
    \Cr{const:Gronwall3} \xi^{\frac{3}{2}} 
   \right)
   }
   \, d\xi
   \\
   &
   =:J_1-J_2.
  \end{split}
 \end{equation}
 Note that the integrand of $J_2$ is positive and integrable on
 $[0, \infty)$,  hence there exists a positive constant
 $\Cl{const:intG<C}>0$ such that $J_2\leq \Cr{const:intG<C}$. From
 \eqref{G's_integrate},  we have
 \begin{equation}
  J_1
   =
   \int_{G(0)}^{G(T)}
   \frac{1}{
   \Cr{const:Gronwall2} 
   \xi^{\frac{3}{2}}
   }
   \, d\xi
   \leq
   \Cr{const:intG<C}
   +
   \frac{2}{\Cr{const:Gronwall1}}.
 \end{equation}
 Compute the integration $J_1$,  we have
 \begin{equation}
  \label{ineq:G}
   G(T)^{-\frac{1}{2}}
   \geq
   G(0)^{-\frac{1}{2}}
   -
   \frac{\Cr{const:intG<C}\Cr{const:Gronwall2}}{2}
   -
   \frac{\Cr{const:Gronwall2}}{\Cr{const:Gronwall1}}.
 \end{equation}
 Here,  we assume that
 \begin{equation*}
  \label{G(0)'sassume}
   g(0)
   =
   G(0)
   <
   \Cr{const:Gronwall_Initial}
   :=
   \left(
     \frac{\Cr{const:intG<C}\Cr{const:Gronwall2}}{2}
     +
     \frac{\Cr{const:Gronwall2}}{\Cr{const:Gronwall1}}.
    \right)^{-2}
 \end{equation*}
 and define
 \begin{equation}
  \Cr{const:Gronwall_Coefficient}
   :=
   \left(
   G(0)^{-\frac{1}{2}}
   -
   \Cr{const:Gronwall_Initial}^{-\frac{1}{2}}
   \right)^{-2}
   >0.
 \end{equation}
 Then,  from \eqref{ineq:G},  we have $ G(T) \le
   \Cr{const:Gronwall_Coefficient}$.
\end{proof}

Now,  we are in a position to demonstrate the main theorem.

\begin{proof}[Proof of Thoerem \ref{theorem}]
 
 Define $g(t)$ by
 \begin{equation}
  g(t):=
   \D[\unk](t)
   =
   \int_\Omega
   |\nabla\mu|^2\unk\, dx.
 \end{equation}
 From \eqref{eq:2.Second_time_derivative_F_Final} and
 $\frac{d}{dt}\F[\rho](t)=-\D[\unk](t)$,  we have
 \begin{equation}
  g(t)
   \leq
   -
   \lambda
   g(t)
   +
   \Cr{const:|nablamu|^2f^3/2}
   \left(g(t)\right)^3
   +
   \Cr{const:|nablamu|^2f^3}
   \left(g(t)\right)^{\frac{3}{2}}.
 \end{equation} 
 Then,  we obtain Theorem \ref{theorem} by applying the Gronwall lemma
 (Lemma \ref{lem:2.Gronwall_Lemma}) with $\Cr{const:Gronwall1}=\lambda, 
 \Cr{const:Gronwall2}=\Cr{const:|nablamu|^2f^3}$,  and
 $\Cr{const:Gronwall3}=\Cr{const:|nablamu|^2f^3/2}$.
\end{proof}

\section{Further study}
As we noted in Remark \ref{rem:2.dependency_constant_Interpolation} and \ref{rem3}, the constants $\Cr{const:f>c}$, $\Cr{const:f<c}$, the lower and upper bound of the solution of \eqref{Nonlinear-Fokker-Planck}, depend on the diffusion coefficient, so on $\Cr{const:D_min}$ too.
Thus, dependency of $\Cr{const:f>c}$, $\Cr{const:f<c}$ on $\Cr{const:D_min}$ should be discussed to apply Theorem \ref{theorem}. 
Also, we should study global-in-time solutions for \eqref{Nonlinear-Fokker-Planck}. We are currently working on this subject and will present it elsewhere.

In Theorem \ref{theorem}, we assumed the dimension restriction $n=1,2,3$ and the largeness of the diffusion coefficient $\Cr{const:D_min}$. It is not clear whether these assumptions are essential. 
The key difficulty about the dimension restriction comes from $|\nabla \mu|^3$, the cubic of the gradient of $\mu$.
We may need some regularity results for \eqref{Nonlinear-Fokker-Planck}. 
The smallness of $\nabla d$ naturally arises from the problem close to the case of the constant diffusion coefficient.
Replacing the largeness of $d$ with the smallness of $\nabla d$, and assuming the other assumptions, we can obtain the exponential decay of $\D[\unknown](t)$.
Thus, $\nabla \log d$ might be key to deriving the convergence of the equilibrium state for \eqref{Nonlinear-Fokker-Planck}.

Finally, we mention the degeneracy of the diffusion in \eqref{Nonlinear-Fokker-Planck}.
Since we assumed the positivity of classical solutions in Theorem \ref{theorem}, we do not treat the degeneracy of the diffusion.
For the homogeneous case, namely the diffusion coefficient $d$ is a constant, as in \cites{MR1853037,MR1777035,MR3497125}, we can handle the degeneracy of the nonlinear diffusion of porous medium type. We essentially use the positivity of the solution to deduce the interpolation inequality.
Proposition \ref{prop:|v|^3} with the weight measure $\unknown\,dx$. It was needed to control the cubic nonlinearity of $\nabla \mu$. 
It is an interesting problem to study long-time asymptotic behavior of weak solutions to \eqref{Nonlinear-Fokker-Planck} to address the degeneracy of the diffusion.

\section*{Acknowledgments}

The work of Masashi Mizuno was partially supported by JSPS KAKENHI Grant
Numbers JP22K03376 and JP23H00085.

\bibliography{references}

\end{document}